\documentclass[a4paper,10pt]{amsart}

\usepackage{amssymb, enumerate, amssymb}
\usepackage{amstext}
\usepackage{amsmath}
\usepackage{amscd}
\usepackage{latexsym}
\usepackage{amsfonts}
\usepackage{amsthm}
\usepackage{fancybox}
\usepackage[all]{xy}

\newtheorem{thm}{Theorem}[section]
\newtheorem{defn}[thm]{Definition}

\newtheorem{lem}[thm]{Lemma}
\newtheorem{prop}[thm]{Proposition}
\newtheorem{cor}[thm]{Corollary}

\newtheorem{rem}[thm]{Remark}

\def\Ker{\mathrm{Ker}}

\def\Q{{\mathbb Q}}

\def\C{{\mathbb C}}

\def\Der{{\rm Der}}
\def\map{{\rm map}}
\def\wl{{\rm WL}}
\def\nil{{\rm nil}}

\numberwithin{equation}{section}

\title{A model for the Whitehead product in rational mapping spaces}
\author{Takahito Naito}
\date{}
\address{Interdisciplinary Graduate School of Science and Technology, Shinshu University, 3-1-1 Asahi, Matsumoto, Nagano 390-8621, Japan}
\email{naito@math.shinshu-u.ac.jp}

\begin{document}
\maketitle

\begin{abstract}
We describe the Whitehead products in the rational homotopy group of a connected component of a mapping space in terms of the Andr${\rm \acute{e}}$-Quillen cohomology. As a consequence, an upper bound for the Whitehead length of a mapping space is given.
\end{abstract}

\section{Introduction}
We assume that all space in this paper are path connected CW-complexes with a nondegenerate base point $*$.
Let $X$ and $Y$ be simply-connected spaces and $\map(X,Y;f)$ the path component of the space of free maps from $X$ to $Y$ containing the based map $f:X\to Y$. We denote by $\Lambda V$ and $B$ a minimal Sullivan model for $Y$ and a model for $X$, respectively. Let $\overline{f}:\Lambda V \to B$ be a model for $f$ and $\Der ^{*}(\Lambda V,B;\overline{f})$ the complex of $\overline{f}$-derivations; see next section for proper definitions and details. The cohomology of $\Der ^{*}(\Lambda V,B;\overline{f})$ is called the Andr$\acute{{\rm e}}$-Quillen cohomology of $\Lambda V$ with coefficients in $B$, denoted by $H^{*}_{{\rm AQ}}(\Lambda V,B;\overline{f})$; see \cite{BL2005}.\\
\indent
Suppose that $X$ is a finite CW-complex. The $n$th rational homotopy group of $\map(X,Y;f)$ is isomorphic to $H_{{\rm AQ}}^{-n}(\Lambda V,B;\overline{f})$ as abelian groups for $n\geq 2$. This fact has been proved by Block and Lazarev \cite{BL2005}, Buijs and Murillo \cite{BM2008}, Lupton and Smith \cite{LS2007}.
Moreover Buijs and Murillo \cite{BM2008} defined a bracket in the Andr${\rm {\acute e}}$-Quillen cohomology $H_{{\rm AQ}}^{*}(\Lambda V,B;\overline{f})$ which coincides with the Whitehead product in $\pi_{*}(\map (X,Y;f))\otimes \Q$. We mention that the isomorphism due to Buijs and Murillo is constructed relying on the Sullivan model for $\map(X,Y;f)$ due to Haefliger \cite{Hae1982} and Brown and Szczarba \cite{BS1997}. To this end, the finiteness of a model $B$ for the source space $X$ is assumed in the result \cite[Theorem 1.3]{BS1997} and also \cite[\S 3]{Hae1982}.\\
\indent
On the other hand, a finiteness hypothesis of $X$ shows that $\pi_{n}(\map (X,Y;f))\otimes \Q$ is isomorphic to $\pi_{n}(\map(X,Y_{\Q};lf))$, where $l:Y\to Y_{\Q}$ the localization map; see \cite[II. Theorem 3.11]{HMR} and \cite[Theorem 2.3]{Sm1996}. Then the isomorphism constructed in \cite{BL2005} and \cite{LS2007} factors as following;
\[
\begin{xy}
(0,0)="A"*{\pi _{n}(\map (X,Y;f))\otimes \Q},
(42,0)="B"*{\pi _{n}(\map (X,Y_{\Q};lf))},
(82,0)="C"*{H_{{\rm AQ}}^{-n}(\Lambda V,B;\overline{f}).},
\ar_{\cong} "A"+/r1.8cm/;"B"+/l1.6cm/,
\ar_{\cong}^{\Phi } "B"+/r1.6cm/;"C"+/l1.3cm/
\end{xy}
\]
The proper definition of $\Phi $ is described in Section 2. By the proof of \cite[Theorem 2.1]{LS2007}, we see that the second map $\Phi $ is an isomorphism without a finiteness hypothesis of $X$. Also the assertion of \cite[Theorem 3.8]{BL2005} is that the map $\Phi $ is an isomorphism.
In this paper, we introduce a bracket in the Andr${\rm {\acute e}}$-Quillen cohomology which coincides with the Whitehead product in $\pi_{*}(\map (X,Y_{\Q};f))$ up to the isomorphism $\Phi $ without assuming that $X$ has a finite dimensional commutative model.\\
\indent
Let $X$ be a simply-connected space with a commutative model $B$ and $Y$ be a $\Q$-local, simply-connected space of finite type. Then we have a model $\overline{f}:\Lambda V\to B$ for a based map $f:X\to Y$.
Now, we define a bracket in $H_{{\rm AQ}}^{*}(\Lambda V,B;\overline{f})$ 
$$
[ \ , \ ]:H_{{\rm AQ}}^{n}(\Lambda V,B;\overline{f})\otimes H_{{\rm AQ}}^{m}(\Lambda V,B;\overline{f})\longrightarrow H_{{\rm AQ}}^{n+m+1}(\Lambda V,B;\overline{f})
$$
by
\begin{multline}\label{(1.1)}
[\varphi ,\psi ](v)=\\
(-1)^{n+m-1}\displaystyle\sum \Bigl( \displaystyle\sum_{i\neq j}(-1)^{\varepsilon _{ij}}\overline{f}(v_{1}\cdots v_{i-1})\varphi (v_{i})\overline{f}(v_{i+1}\cdots v_{j-1})\psi (v_{j})\overline{f}(v_{j+1}\cdots v_{s})\Bigr),
\end{multline}
where $v$ is a basis of $V$, $dv=\sum v_{1}v_{2}\cdots v_{s}$ and
$$ \varepsilon _{ij}=
\left\{
\begin{array}{lc}
|\varphi |(\displaystyle\sum_{k=1}^{i-1}|v_{k}|)+|\psi |(\displaystyle\sum_{k=1}^{j-1}|v_{k}|)+|\varphi ||\psi |&(i<j)\\
|\varphi |(\displaystyle\sum_{k=1}^{i-1}|v_{k}|)+|\psi |(\displaystyle\sum_{k=1}^{j-1}|v_{k}|)&(j<i)
\end{array}
\right.$$
The following is our main result of this paper.
\begin{thm}\label{thm:1.1}
The isomorphism $\Phi : \pi _{n}(\map(X,Y;f))\rightarrow H_{{\rm AQ}}^{-n}(\Lambda V,B;\overline{f})$ is compatible with the Whitehead product in $\pi _{n}(\map(X,Y;f))$ and the bracket in $H_{{\rm AQ}}^{-n}(\Lambda V,B;\overline{f})$ defined by the formula \eqref{(1.1)}.
\end{thm}
If $X$ is finite, then the bracket in $H_{{\rm AQ}}^{*}(\Lambda V,B;\overline{f})$ coincides with that due to Buijs and Murillo \cite{BM2008} up to sign. Thus Theorem \ref{thm:1.1} is regarded as a generalization of \cite[Theorem 2]{BM2008}.
Let $\map_{*}(X,Y;f)$ be the path-component of the space of based maps form $X$ to $Y$ containing the based map $f:X\to Y$. We apply the same argument to the case of the based mapping space $\map_{*}(X,Y;f)$; see the last of Section 3 for details.\\
\indent
As an application of the main result, we study the Whitehead length of a mapping space. 
{\it The Whitehead length of a space $Z$}, written $\wl (Z)$, is the length non-zero iterated Whitehead products in $\pi_{\geq 2}(Z)$.
In \cite{LS}, Lupton and Smith give some results and examples related to a Whitehead length of mapping spaces $\map(X,Y;f)$ using a Quillen model.
We will give another proof of their results using the bracket in the Andr${\rm {\acute e}}$-Quillen cohomology; see Proposition \ref{prop:4.1}.
To give an upper bound for the Whitehead length of $\map_{*}(X,Y;f)$, we introduce a numerical invariant. 
\begin{defn}[{\cite[p315]{Rat}}]
{\rm
Let $A$ be a connected graded algebra.} The product length of $A$, {\rm written $\nil A$, is the greatest integer $n$ such that $A^{+}A^{+}\cdots A^{+}\neq 0$ ($n$ factors).}
\end{defn}
In \cite{Bu2010}, Buijs proved the following theorem.
\begin{thm}[{\cite[Theorem 0.3]{Bu2010}}]\label{thm:4.5}
Let $X$ and $Y$ be simply-connected spaces with finite type over $\Q$ and $B$ a model for $X$. If $\wl (Y_{\Q})=1$, then
$$
\wl (\map_{*}(X,Y;f)_{\Q})\leq \nil B -1.
$$
\end{thm}
Using the bracket in the Andr${\rm {\acute e}}$-Quillen cohomology, we can prove the following proposition, which refines the above result in the case when the source space of the mapping space is finite CW complex; see Remark \ref{rem:4.7}.
\begin{prop}\label{prop:4.6}
Let  $X$ and $Y$ be simply-connected spaces with finite type over $\Q$ and $B$ a model for $X$. Assume further that $Y$ is $\Q$-local. Then, we have
$$\wl (\map_{*}(X,Y;f))\leq \nil B.$$
Moreover, if $\wl(Y)=1$ and $\nil B\geq 2$, then
$$
\wl (\map_{*}(X,Y;f))\leq \frac{1}{\omega  -1}(\nil B -1)+1,
$$
where $\omega  = {\rm min}\{ n\geq 2 \ | \ d(V)\subset \Lambda ^{\geq n} V\}$.
\end{prop}
We will also compute the Whitehead length of mapping spaces $\map (\C P^{\infty}\times \C P^{m}, \C P^{\infty}_{\Q}\times \C P^{n}_{\Q};f)$.\\
\indent
The organization of this paper is as follows. In Section 2, we will recall several
fundamental results on rational homotopy theory. The isomorphism $\Phi $ in \cite{BL2005} and \cite{LS2007} is also described. In Section 3, we prove Theorem \ref{thm:1.1}. To this end, a model for the Whitehead product of mapping spaces will be constructed in the section. The Whitehead length of mapping spaces is considered in Section 4. A computational example of the Whitehead length is presented in Section 5.

\section{Preliminaries}
We refer the reader to the book \cite{Rat} for the fundamental facts on rational homotopy theory.
A {\it Sullivan algebra} is a free commutative differential graded algebra over the field of rational numbers $\Q$ (or simply CDGA in this paper), $(\Lambda V ,d)$, with a $\Q$-graded vector space $V=\bigoplus _{i\geq 1}V^{i}$ where $V$ has an increasing sequence of subspaces $V(0)\subset V(1)\subset \cdots $ which satisfy the conditions that $V=\bigcup_{i\geq 0}V(i)$, $d=0$ in $V(0)$ and $d:V(i)\to \Lambda V(i-1)$ for any $i\geq 1$.\\
\indent
We recall a {\it minimal Sullivan model} for a simply-connected space $X$ with finite type. It is a Sullivan algebra of the form $(\Lambda V,d)$ with $V=\bigoplus _{i\geq 2}V^{i}$ where each $V^{i}$ is of finite dimension and $d$ is decomposable; that is, $d(V)\subset \Lambda ^{\geq 2}V$. 
Moreover, $(\Lambda V,d)$ is equipped with a quasi-isomorphism $(\Lambda V,d) \stackrel{\simeq }{\longrightarrow } A_{{\rm PL}}(X)$ to the CDGA $A_{{\rm PL}}(X)$ of differential polynomial forms on $X$. Observe that, as algebras, $H^{*}(\Lambda V,d)\cong H^{*}(A_{{\rm PL}}(X))\cong H^{*}(X;\Q)$. 
For instance, a minimal Sullivan model for the $n$-sphere $S^{n}$, $M(S^{n})$, is the form $(\Lambda (e_{n}),0)$ if $n$ is odd and $(\Lambda (e_{n},e_{2n-1}),de_{2n-1}=e_{n}^{2})$ if $n$ is even, where $|e_{n}|=n$ and $|e_{2n-1}|=2n-1$. \\
\indent
A {\it model} for a space $X$ is a connected CDGA $(B,d)$ such that there is a quasi-isomorphism from a minimal Sullivan model for $X$ to $B$.
The two maps of CDGA $\varphi_{1}$ and $\varphi _{2}$ from a Sullivan algebra $\Lambda V$ to a CDGA $A$ are {\it homotopic} if there exists a CDGA map $H:\Lambda V\to A\otimes \Lambda (t,dt)$ such that $(1\cdot \varepsilon _{i})H=\varphi _{i}$ for $i=0,1$. Here, $\Lambda (t,dt)$ is the free CDGA with $|t|=0$, $|dt|=1$ and the differential $d$ of $\Lambda (t,dt)$ sends $t$ to $dt$. The map $\varepsilon _{i}:\Lambda (t,dt)\to \Q$ defined by $\varepsilon _{i}(t)=i$. Denote $[\Lambda V,A]$ by the set of homotopy classes of CDGA maps from $\Lambda V$ to $A$. \\
\indent
Let $f:X\to Y$ be a map between spaces of finite type. Then there exists a CDGA map $\widetilde{f}$ from a minimal Sullivan model $(\Lambda V_{Y},d)$ for $Y$ to a minimal Sullivan model $(\Lambda V_{X},d)$ for $X$ which makes the diagram
\[\xymatrix{
A_{{\rm PL}}(Y) \ar[rr]^{A_{{\rm PL}}(f)} & & A_{{\rm PL}}(Y)\\
\Lambda V_{Y} \ar[u]^{\simeq } \ar[rr]_{\widetilde{f}} & & \Lambda V_{X} \ar[u]_{\simeq}
}\]
commutative up to homotopy. Let $\rho :\Lambda V_{X}\stackrel{\simeq}{\longrightarrow }B$ a model for $X$, we call $\rho \widetilde{f}$ a model for $f$ associated with models $\Lambda V_{Y}$ and $B$ and denote it by $\overline{f}$.\\
\indent 
We use the following result when constructing a model for the Whitehead product of a mapping space.
\begin{prop}[{\cite[Proposition 12.9]{Rat}}] \label{liftting lemma}
Let $A$ and $C$ be CDGAs, $\Lambda V$ a Sullivan algebra and $\pi:A\to C$ a quasi-isomorphism.
Then the map
$$
\pi _{*}:[\Lambda V,A]\longrightarrow [\Lambda V,C]
$$
induced by $\pi$ is bijective.
\end{prop}
\begin{rem}\label{rem:2.2}{\rm
If $\pi $ is a surjective quasi-isomorphism, we can construct a CDGA map $\phi :\Lambda V\to A$ such that $\pi \phi =\psi$ for any CDGA map $\psi:\Lambda V\to C$ by induction on a degree of $V$ (\cite[Lemma 12.4]{Rat}). Assume that $\phi d(v)$ is defined for any basis $v$ in $V$. Since $\pi $ is a surjective quasi-isomorphism and $\pi \phi d(v)=d\psi (v)$, we can find $a\in A$ such that $d(a)=\phi d(v)$ and $\pi (a)=\psi (v)$. Then, we extend $\phi $ with $\phi (v)=a$.
}\end{rem}
\indent
We next recall the definition of the Whitehead product. Let $\alpha \in \pi_{n}(X)$ and $\beta \in \pi_{m}(X)$ be elements represented by $a:S^{n}\to X$ and $b:S^{m}\to X$, respectively. Then the Whitehead product $[\alpha ,\beta ]_{w}$ is defined to be the homotopy class of composite
\[
\begin{xy}
(-25,0)="A"*{S^{n+m-1}}, (0,0)="B"*{S^{n}\vee S^{m}}, (28,0)="C"*{X},
\ar ^{\eta}"A"+/r0.7cm/;"B"+/l0.7cm/, \ar ^{\nabla (a\vee b)}"B"+/r0.7cm/;"C"+/l0.3cm/
\end{xy}
\]
where $\eta$ is the universal example and $\nabla :X\vee X\to X$ is the folding map.
Recall that the differential $d$ of $\Lambda V$ can be written by $d=\sum _{i\geq 1}d_{i}$ with $d_{i}(V)\subset \Lambda ^{i+1}V$. The map $d_{1}$ is called the {\it quadratic part} of $d$.
We see that the quadratic part $d_{1}$ is related with the Whitehead products in $\pi_{*}(X)$.
We denote by $Q(g)^{n}:V^{n}\to \Q e_{n}$ the linear part of a model for $g$, $\overline{g}:\Lambda V\to M(S^{n})$. Define a paring and a trilinear map
\begin{align*}
&\langle \ ; \ \rangle :V\times \pi _{*}(X) \longrightarrow \Q ,\\
&\langle \ ; \ , \ \rangle :\Lambda ^{2}V\times \pi _{*}(X) \times \pi _{*}(X) \longrightarrow \Q 
\end{align*}
by
$$
\langle v ; \alpha  \rangle e_{n}=
\left\{
\begin{array}{cc}
Q(g)^{n}v&(| v |=n)\\
0&(| v | \neq n)
\end{array}
\right.
$$
and
$$
\langle vw;\alpha ,\beta \rangle=\langle v;\alpha \rangle \langle w;\beta \rangle +(-1)^{| w | | \alpha |} \langle w;\alpha \rangle \langle v;\beta \rangle ,
$$
respectively.
\begin{prop}[{\cite[Proposition 13.16]{Rat}}] \label{w.p}
The following holds
$$
\langle d_{1}v;\alpha ,\beta \rangle=(-1)^{n+m-1}\langle v;[\alpha ,\beta ]_{w}\rangle ,
$$
where $v\in V$, $\alpha \in \pi_{n}(X)$, $\beta \in \pi_{m}(X)$.
\end{prop}
We conclude this section by recalling the isomorphism $\Phi $ defined in \cite{BL2005} and \cite{LS2007} from $\pi_{n}(\map(X,Y;f))$ to $H_{{\rm AQ}}^{-n}(\Lambda V,B;\overline{f})$ in the setting of a simply-connected space $X$ and a $\Q$-local, simply-connected space $Y$ with finite type.
We here recall the complex of $\overline{f}$-derivations from $\Lambda V$ to B which denoted by $\Der^{*}(\Lambda V,B;\overline{f})$. An element $\theta \in \Der^{n}(\Lambda V,B;\overline{f})$ is a $\Q$-linear map of degree $n$ with $\theta (xy)=\theta (x)\overline{f}(y)+(-1)^{n|x|}\overline{f}(x)\theta (y)$ for any $x,y\in \Lambda V$. The differentials $\partial :\Der ^{n}(\Lambda V,B;\overline{f})\to \Der^{n+1}(\Lambda V,B;\overline{f})$ are defined by $\partial (\theta )=d\theta -(-1)^{n}\theta d$.\\
\indent
Let $\alpha \in \pi_{n}(\map(X,Y;f))$ and $g:S^{n}\times X\to Y$ the adjoint of $\alpha $. We note that $g$ satisfy $g|_{X}=f$. Then there exists a model $\overline{g}:\Lambda V\to M(S^{n})\otimes B$ for $g$ such that the following diagram is strictly commutative;
\[
\begin{xy}
(0,0)="A"*{\Lambda V},
(30,0)="B"*{M(S^{n})\otimes B},
(15,-12)="C"*{B,},
\ar^{\overline{g}} "A"+/r0.3cm/;"B"+/l1cm/,
\ar_{\overline{f}} "A"+/d0.2cm/+/r0.2cm/;"C"+/u0.2cm/+/l0.2cm/,
\ar^{\varepsilon \cdot 1} "B"+/d0.2cm/+/l0.2cm/;"C"+/u0.2cm/+/r0.2cm/
\end{xy}
\]
where $\varepsilon :M(S^{n})\to \Q$ is the augmentation; see Lemma \ref{lem:3.1}. Since $S^{n}$ is formal, there is a quasi-isomorphism $\phi :M(S^{n})\to (H^{*}(S^{n};\Q),0)$ and, for any $v\in \Lambda V$, we may write
$$
(\phi \otimes 1 )\overline{g}(v)=1\otimes \overline{f}(v)+e_{n}\otimes \theta (v).
$$
Then we see that $\theta $ is a $\overline{f}$-derivation of degree $-n$ and also a cycle in $\Der^{*}(\Lambda V,B;\overline{f})$. Put $\Phi (\alpha ) = \theta $.
\begin{thm}[{\cite[Theorem 3.8]{BL2005} \cite[Theorem 2.1]{LS2007} }]\label{thm:2.3}
The map
$$
\Phi :\pi _{n}(\map(X,Y;f)) \longrightarrow  H_{{\rm AQ}}^{-n}(\Lambda V,B;\overline{f})\\
$$
is an isomorphism of abelian groups for any $n\geq 2$.
\end{thm}


\section{A model for the adjoint of the Whitehead product}

We retain the notation and terminology described in the previous section.
In order to consider the image of the Whitehead product in $\pi_{*}\map(X,Y;f)$ by the isomorphism $\Phi$, we construct an appropriate model for the adjoint of the Whitehead product. This is the key to proving Theorem \ref{thm:1.1}.
Let $X$ be a simply-connected space, $Y$ a $\Q$-local, simply-connected space of finite type and  $f:X\to Y$ a based map.
We denote by $(\Lambda V,d)$ and $(B,d)$ a minimal Sullivan model for $Y$ and a model for $X$, respectively. Let $\overline{f}:\Lambda V\to B$ be a model for $f$ associated with such the models.\\
\indent
We prepare for proving Theorem \ref{thm:1.1}. We see that a minimal Sullivan model for $S^{n}\vee S^{m}$ has the form
$$
M(S^{n}\vee S^{m})=(M(S^{n})\otimes M(S^{m})\otimes \Lambda (\iota_{n+m-1},x_{1},x_{2},\cdots ),d)
$$
in which $d\iota_{n+m-1}=e_{n}e_{m}$ and $|\iota_{n+m-1}|=n+m-1 < |x_{i}|$ for any $i\geq 1$; see \cite[p177]{Rat}.


\begin{lem}\label{lem:3.1}
Let $g:S^{n}\times X\longrightarrow Y$ be a map with $g|_{X}=f$. Then there exists a model $\overline{g}$ for $g$ such that the diagram is strictly commutative;
\[
\begin{xy}
(0,0)="A"*{\Lambda V},
(40,0)="B"*{M(S^{n})\otimes B},
(20,-10)="C"*{B,},
\ar^{\overline{g}} "A"+/r0.3cm/;"B"+/l1cm/,
\ar_{\overline{f}} "A"+/d0.2cm/+/r0.2cm/;"C"+/u0.2cm/+/l0.2cm/,
\ar^{\varepsilon \cdot 1} "B"+/d0.2cm/+/l0.2cm/;"C"+/u0.2cm/+/r0.2cm/
\end{xy}
\]
where $\varepsilon :M(S^{n})\to \Q$ is the augmentation. Moreover, if $g$ satisfy $g|_{X}=f$ and $g|_{S^{n}}=*$, where $*:S^{n}\to Y$ is the constant map to the base point, then there is a model $\overline{g}$ for $g$ such that the following diagram commute strictly;
\[
\begin{xy}
(20,10)="D"*{M(S^{n})},
(0,0)="A"*{\Lambda V},
(40,0)="B"*{M(S^{n})\otimes B},
(20,-10)="C"*{B,},
\ar^{\overline{g}} "A"+/r0.3cm/;"B"+/l1cm/,
\ar_{\overline{f}} "A"+/d0.2cm/+/r0.2cm/;"C"+/u0.2cm/+/l0.2cm/,
\ar^{\varepsilon \cdot 1} "B"+/d0.2cm/+/l0.2cm/;"C"+/u0.2cm/+/r0.2cm/,
\ar^{u\varepsilon } "A"+/u0.2cm/+/r0.2cm/;"D"+/d0.2cm/+/l0.2cm/,
\ar_{1\cdot \varepsilon } "B"+/u0.2cm/+/l0.2cm/;"D"+/d0.2cm/+/r0.2cm/
\end{xy}
\]
where $u:\Q \to M(S^{n})$ is the unit map.
\end{lem}
\begin{proof}
Let $\overline{g}'$ be a model for $g$. We define the map $\overline{g}:\Lambda V\to M(S_{n})\otimes B$ by
$$
\overline{g}(v)=1\otimes (\overline{f}-(\varepsilon \cdot 1)\overline{g}')(v)+\overline{g}'(v).
$$
Then $\overline{g}$ and $\overline{g}'$ are homotopic. Indeed, $\overline{f}$ and $pr\circ \overline{g}'$ are homotopic and 
let $H:\Lambda V\longrightarrow B\otimes \Lambda (t,dt)$ be a its homotopy. Then, the map $\overline{H}:\Lambda V\longrightarrow M(S^{n})\otimes B\otimes \Lambda (t,dt)$ defined by
$$
\overline{H}(v)=1\otimes H(v)+\overline{g}'(v)\otimes 1-1\otimes (\varepsilon \cdot 1) \overline{g}'(v)\otimes 1
$$
is a homotopy from $\overline{g}'$ to $\overline{g}$. A similar argument shows the second assertion.
\end{proof}

Given $\alpha \in \pi_{n}(\map(X,Y;f))$ and $\beta \in \pi_{m}(\map(X,Y;f))$. Let $g:S^{n}\times X\to Y$ and $h:S^{m}\times X\to Y$ be the adjoint maps of $\alpha $ and $\beta $, respectively. In order to consider the image of $[\alpha ,\beta ]_{w}$ by $\Phi$, we construct a model for the adjoint of $[\alpha ,\beta ]_{w}$
\[
\begin{xy}
(0,0)="A"*{ad([\alpha ,\beta ]_{w}):S^{n+m-1}\times X},  (50,0)="B"*{(S^{n}\vee S^{m})\times X}, (80,0)="C"*{Y,}
\ar ^{\eta\times 1} "A"+/r2cm/;"B"+/l1.3cm/, \ar ^{(g|h)} "B"+/r1.3cm/;"C"+/l0.3cm/
\end{xy}
\]
where $(g|h)$ is a map defined by $(g|h)(u_{n},x)=g(u_{n},x)$ and $(g|h)(u_{m},x)=h(u_{m},x)$ for any $u_{n}\in S^{n}$, $u_{m}\in S^{m}$ and $x\in X$.
It is readily seen that the canonical map
$$
\pi :M(S^{n}\vee S^{m})\longrightarrow M(S^{n})\times _{\Q} M(S^{m})
$$
is a surjective quasi-isomorphism, where $M(S^{n})\times _{\Q} M(S^{m})$ is the pull-back of the augmentations $M(S^{n})\to \Q$ and $M(S^{m})\to \Q$. By Proposition \ref{liftting lemma}, we have the following homotopy commutative square
\[\xymatrix{
A_{{\rm PL}}(S^{n}\vee S^{m}) \ar[rrr]^-{(A_{{\rm PL}}(i_{1}),A_{{\rm PL}}(i_{2}))} & & &A_{{\rm PL}}(S^{n})\times _{\Q}A_{{\rm PL}}(S^{m})\\
M(S^{n}\vee S^{m}) \ar[u]^{\simeq} \ar[rrr]_-{\pi}  & && M(S^{n})\times _{\Q} M(S^{m}), \ar[u]_{\simeq}
}\]
where $i_{1}:S^{n}\to S^{n}\vee S^{m}$ and $i_{2}:S^{m}\to S^{n}\vee S^{m}$ are the inclusions.
The commutative diagram
\begin{equation}\label{(3.1)}
\begin{split}
\begin{xy}
(0,0)="A"*{(S^{n}\vee S^{m})\times X}, (35,0)="B"*{S^{m}\times X}, (-35,0)="C"*{S^{n}\times X}, (0,-13)="D"*{Y},
\ar _{i_{2}\times 1}"B"+/l0.8cm/;"A"+/r1.2cm/, \ar ^{i_{1}\times 1} "C"+/r0.8cm/;"A"+/l1.2cm/, \ar _{(g|h)}"A"+/d0.2cm/;"D"+/u0.2cm/, \ar ^{h}"B"+/d0.2cm/;"D"+/r0.2cm/, \ar _{g}"C"+/d0.2cm/;"D"+/l0.2cm/
\end{xy}
\end{split}
\end{equation}
enables us to give the following homotopy commutative diagram:
\begin{equation}\label{(3.2)}
\begin{split}
\xymatrix{
& & & M(S^{n}\vee S^{m})\otimes B \ar[d]^{\pi \otimes 1}\\
\Lambda V \ar[rrr]_-{(\overline{g},\overline{h})} \ar[rrru]^{\overline{(g|h)}} & & & (M(S^{n})\times _{\Q} M(S^{m}))\otimes B,
}
\end{split}
\end{equation}
where $(\overline{g},\overline{h})$ is the map defined by $(\overline{g},\overline{h})(v)=-1\otimes \overline{f}(v)+(j_{1}\otimes 1)\overline{g}(v)+(j_{2}\otimes 1)\overline{h}(v)$ for any $v\in V$ and $j_{1}:M(S^{n})\to M(S^{n})\times _{\Q}M(S^{m})$ and $j_{2}:M(S^{m})\to M(S^{n})\times _{\Q}M(S^{m})$ are the inclusion.
Indeed, by the diagram \eqref{(3.1)}, we see that the diagram
\[
\begin{xy}
(-45,0)="B"*{M(S^{n})\otimes B}, (0,0)="A"*{(M(S^{n})\times _{\Q}M(S^{m}))\otimes B}, (45,0)="C"*{M(S^{m})\otimes B}, (0,-13)="D"*{\Lambda V},
\ar _{p_{1}\otimes 1}"A"+/l2cm/;"B"+/r1cm/, \ar ^{p_{2}\otimes 1}"A"+/r2cm/;"C"+/l1cm/, \ar ^{(\pi \otimes 1)\overline{(g|h)}}"D"+/u0.2cm/;"A"+/d0.2cm/, \ar ^{\overline{g}} "D"+/l0.3cm/;"B"+/d0.3cm/, \ar _{\overline{h}} "D"+/r0.3cm/;"C"+/d0.3cm/
\end{xy}
\]
is homotopy commutative, where $p_{1}$ and $p_{2}$ are the projection. Let $H_{1}$ and $H_{2}$ be homotopies from $(p_{1}\pi \otimes 1)\overline{(g|h)}$ to $\overline{g}$ and from $(p_{2}\pi \otimes 1)\overline{(g|h)}$ to $\overline{h}$, respectively. Then, a CDGA map $H:\Lambda V\to (M(S^{n})\times_{\Q} M(S^{m}))\otimes B \otimes \Lambda (t,dt)$ defined by
$$
H(v)=-1\otimes \overline{f}(v)\otimes 1 + (j_{1}\otimes 1\otimes 1)H_{1}(v) +(j_{2}\otimes 1\otimes 1)H_{2}(v)
$$
for any $v\in V$ is a homotopy from $(\pi \otimes 1)\overline{(g|h)}$ to $(\overline{g},\overline{h})$.
If there is a map $\phi :\Lambda V\to M(S^{n}\vee S^{m})\otimes B $ such that $(\pi \otimes 1)\phi =(\overline{g},\overline{h})$, $\phi $ and $\overline{(g|h)}$ is homotopic by Proposition \ref{liftting lemma}. Therefore,
it is only necessary to construct of a lift $\phi$ of the diagram \eqref{(3.2)} for getting a model for $(g|h)$. 
\begin{lem}\label{lem:3.3}
There is a model $\phi $ for $(g|h)$ such that for any $v\in V$, $\phi (v)$ has no term of the form $e_{n}e_{m}\otimes u$ for some $u \in B$ and the following diagram commutes strictly;
\[
\begin{xy}
(0,0)="A"*{\Lambda V},
(40,0)="B"*{M(S^{n}\vee S^{m})\otimes B},
(15,-12)="C"*{B.},
\ar^{\phi } "A"+/r0.3cm/;"B"+/l1.5cm/,
\ar_{\overline{f}} "A"+/d0.2cm/+/r0.2cm/;"C"+/u0.2cm/+/l0.2cm/,
\ar^{\varepsilon \cdot 1} "B"+/d0.2cm/+/l0.2cm/;"C"+/u0.2cm/+/r0.2cm/
\end{xy}
\]
\end{lem}
\begin{proof}
First, we recall the construction of a lift $\phi '$ in Remark \ref{rem:2.2}. For any basis $v$ of $V$, we can find $a\in M(S^{n}\vee S^{m})\otimes B$ so that $da=\phi 'dv$ and $(\pi \otimes 1)a=(\overline{g},\overline{h}) v$. We may write
$$
a=1\otimes \overline{f}(a)+e_{n}\otimes a_{2}+e_{m}\otimes a_{3}+ \iota_{n+m-1}\otimes a_{4}+e_{n}e_{m}\otimes a_{5}+{\mathcal O}_{a},
$$
where $a_{i}\in B$ and ${\mathcal O}_{a}$ denote other terms. We put
\begin{equation}\label{(3.3)}
a'=1\otimes \overline{f}(a)+e_{n}\otimes a_{2}+e_{m}\otimes a_{3}+ \iota_{n+m-1}\otimes (a_{4}+da_{5})+{\mathcal O}_{a}.
\end{equation}
Then it follows that $d(a)=d(a')$ and $(\pi \otimes 1)(a)=(\pi \otimes 1)(a')$. Hence, the map $\phi $ defined by
$$
\phi (v)=a'
$$
satisfies the condition that $(\pi \otimes 1)\phi =(\overline{g},\overline{h})$. Thus we see that $\phi$ is a model for $(g|h)$. The second assertion is shown using the equation \eqref{(3.3)}.
\end{proof}
Combining these results we prove our main result.\\
\noindent
{\it Proof of Theorem \ref{thm:1.1}}.
Given $\alpha \in \pi _{n}(\map(X,Y;f))$ and $\beta \in \pi _{m}(\map(X,Y;f))$. Let $g:S^{n}\times X\to Y$ and $h:S^{m}\times X\to Y$ be the adjoint maps of $\alpha $ and $\beta $, respectively. First, by the proof of Proposition \ref{w.p}, we see that a model $\overline{\eta }$ for the universal example $\eta$ sends $\iota _{n+m-1}\in M(S^{n}\vee S^{m})$ to $(-1)^{n+m-1}e_{n+m-1}\in M(S^{n+m-1})$. We choose a model
$\phi $ for the map $(g|h)$ as in Lemma \ref{lem:3.3}. We may write, modulo the ideal generated by elements of $M(S^{n}\vee S^{m})$ of degree greater than $n+m-1$ and generators $e_{2n-1}$ and $e_{2m-1}$ if there exists,
\begin{align*}
&\phi (v)\equiv 1\otimes \overline{f}(v) + e_{n}\otimes u_{2} +e_{m}\otimes u_{3} +\iota _{n+m-1}\otimes u_{4},\\
&\phi (v_{i})\equiv 1\otimes \overline{f}(v_{i}) + e_{n}\otimes u_{i2}+ e_{m}\otimes u_{i3}+\iota _{n+m-1}\otimes u_{i4} 
\end{align*}
for any $v\in V$ and $dv=\sum v_{1}v_{2}\cdots v_{s}$. Since, $(\overline{\eta }\otimes 1)\phi (v)=1\otimes \overline{f}(v)+e_{n+m-1}\otimes (-1)^{n+m-1}u_{4}$, it follows that $\Phi( [\alpha ,\beta ]_{w})(v)=(-1)^{n+m-1}u_{4}$. On the other hand, $\phi $ is a CDGA map and satisfies the condition of Lemma \ref{lem:3.3}. We then have
\begin{align*}
&e_{n}e_{m}\otimes u_{4}=\\
& \ \ \ e_{n}e_{m}\otimes\displaystyle\sum \Bigl( \displaystyle\sum_{i\neq j}(-1)^{\varepsilon _{ij}}\overline{f}(v_{1}\cdots v_{i-1})u_{i2}\overline{f}(v_{i+1}\cdots v_{j-1})u_{j3}\overline{f}(v_{j+1}\cdots v_{s})\Bigr).
\end{align*}
By commutativity of the diagram \eqref{(3.2)}
and the definition of $\Phi $, we see that $u_{i2}=\Phi (\alpha )(v_{i})$ and $u_{j3}=\Phi (\beta )(v_{j})$. Therefore,
$
\Phi ([\alpha ,\beta ]_{w})(v)=(-1)^{n+m-1}u_{4}=[\Phi (\alpha ),\Phi (\beta )](v).
$
This completes the proof.
\hfill\qed
\\
\ \\
\indent
In the rest of this section, we also consider the Whitehead product in a based mapping space $\map_{*}(X,Y;f)$.
Given $\alpha \in \pi _{n}(\map_{*}(X,Y;f))$ and let $g:S^{n}\times X\to Y$ be the adjoint map of $\alpha $. Since $g$ satisfy $g|_{X}=f$ and $g|_{S^{n}}=*$, by Lemma \ref{lem:3.1}, there exists a model for $g$, $\overline{g}$, such that $(\varepsilon \cdot 1)\overline{g}=\overline{f}$ and $(1\cdot \varepsilon )\overline{g}=u\varepsilon $. The second equation shows that $\Phi (\alpha )$ is a $\overline{f}$-derivation of degree $-n$ from $\Lambda V$ to the augmentation ideal $B^{+}$ of $B$. We then get the map of abelian groups
$$
\Phi ': \pi_{n}(\map_{*}(X,Y;f))\longrightarrow H^{-n}_{{\rm AQ}}(\Lambda V,B^{+};\overline{f}); \ \Phi '(\alpha )=\Phi (\alpha )
$$
for $n\geq 2$ and a straight-forward modification of Theorem \ref{thm:2.3} shows the following proposition..
\begin{prop}
The map $\Phi' : \pi_{n}(\map_{*}(X,Y;f))\to H^{-n}_{{\rm AQ}}(\Lambda V,B^{+};\overline{f})$ is an isomorphism for $n\geq 2$.
\end{prop}

This proposition also enables us to get the following corollary.

\begin{cor}\label{cor:3.5}
The restriction of the bracket defined by \eqref{(1.1)} in $H_{{\rm AQ}}^{*}(\Lambda V,B;\overline{f})$ to $H_{{\rm AQ}}^{*}(\Lambda V,B^{+};\overline{f})$ corresponds the Whitehead product in $\pi _{*}(\map_{*}(X,Y;f))$ via the above isomorphism  $\Phi $ from $\pi _{n}(\map_{*}(X,Y;f))$ to $H_{{\rm AQ}}^{-n}(\Lambda V,B^{+};\overline{f})$.
\end{cor}

\begin{proof}
Given $\alpha \in \pi _{n}(\map(X,Y;f))$ and $\beta \in \pi _{m}(\map(X,Y;f))$. Since $\varepsilon \Phi '(\alpha )=0$ and $\varepsilon \Phi '(\beta )=0$, it follows that $\varepsilon \Phi' ([\alpha ,\beta ]_{w})=0$ by the formula \eqref{(1.1)}.
\end{proof}

\section{The Whitehead length of mapping spaces}

In this section, we consider the Whitehead length of mapping spaces. We recall the definition of the Whitehead length; see Section 1.
By the definition, $\wl (Z)=1$ means that all Whitehead products vanish. Now we consider a upper bound of $\wl (\map (X,Y;f))$. The following result is proved by Lupton and Smith.

\begin{prop}[{\cite[Theorem 6.4]{LS}}]\label{prop:4.1}
Let  $X$ and $Y$ be $\Q$-local, simply-connected spaces with finite type. If $Y$ is coformal; that, is a minimal Sullivan model for $Y$ of the form $(\Lambda V,d_{1})$, then
$$
\wl(\map(X,Y;f))\leq \wl (Y).
$$
\end{prop}
We give another proof of Proposition \ref{prop:4.1} using the bracket defined by Theorem \ref{thm:1.1}. Before proving the proposition, we introduce a numerical invariant which is called the $d_{1}$-depth for a simply-connected space $Z$ and recall the relationship between the Whitehead length and the $d_{1}$-depth.

\begin{defn}{\rm
Let $(\Lambda V,d)$ be a minimal Sullivan model for a simply-connected space $Z$ and $d_{1}$ the quadratic part of $d$. } The $d_{1}$-depth of $Z$, {\rm denoted by $d_{1}$-depth($Z$), is the greatest integer $n$ such that $V_{n-1}$ is a proper subspace of $V_{n}$ with
$$
V_{-1}=0, \ V_{0}=\{v\in V \ | \ d_{1}v=0\} \ \text{and} \ V_{i}=\{v\in V \ | \ d_{1}v\in \Lambda V_{i-1}\} \ (i\geq 1).
$$
}
\end{defn}

\begin{thm}[{\cite[Theorem 4.15]{Kaj2005}\cite[Theorem 2.5]{KY2006}}]\label{thm:4.2}
Let $Y$ be a $\Q$-local, simply-connected space. Then $d_{1}$-{\rm depth}$(Y)+1= \wl (Y)$.
\end{thm}

\noindent
{\it Proof of Proposition \ref{prop:4.1}}.
Let $m=d_{1}$-depth$(Y)$. If 
$$0=d^{m+1}(v)=d_{1}^{m+1}(v)=\sum u_{1}u_{2}\cdots u_{s}
$$
for any $v\in V$, then some $u_{i}$ are zero by the definition of $d_{1}$-depth. It follows that, for any $\varphi _{1},\varphi _{2},\dots , \varphi _{m+2}\in H^{\leq -2}_{{\rm AQ}}(\Lambda V,B;\overline{f})$, 
$$
[\varphi _{1},[\varphi _{2},\cdots [\varphi _{m+1},\varphi _{m+2}] \cdots]](v)=0.
$$
Hence, by Theorem \ref{thm:1.1} and Theorem \ref{thm:4.2}, we have $\wl (\map(X,Y;f))\leq m+1 = \wl (Y)$.
\hfill\qed

We next prove Proposition \ref{prop:4.6}.




\noindent
{\it Proof of Proposition \ref{prop:4.6}}.
Let $m=\wl (\map_{*}(X,Y;f))$. If $m=1$, then the assertion is trivial, so we may assume that $m\geq 2$. By Corollary \ref{cor:3.5}, there are elements $\varphi _{1},\varphi _{2},\cdots,\varphi _{m}$ in $H^{\leq -2}_{{\rm AQ}}(\Lambda V,B^{+};\overline{f})$ such that 
$$
[\varphi _{1},[\varphi _{2},\cdots ,[\varphi _{m-1},\varphi _{m}]\cdots ]](v)\neq 0
$$
for some $v\in V$. The formula $(1,1)$ yields that $\nil B\geq m$. Moreover, if $\wl(Y)=1$, then $V=\Ker d_{1}$ by Theorem \ref{thm:4.2}. It means that if $d^{m-1}(v)=\sum v_{1}v_{2}\cdots v_{s}$, then $s\geq (m-2)(\omega -1)+\omega $. Then we see that
$$\nil B\geq (m-2)(\omega -1)+\omega \geq \omega
$$
since $[\varphi _{1},[\varphi _{2},\cdots ,[\varphi _{m-1},\varphi _{m}]\cdots ]](v)\neq 0$. Therefore, we have $m \leq \frac{1}{\omega  -1}(\nil B -1)+1.$
\hfill\qed\\
\begin{rem}\label{rem:4.7}{\rm
Suppose that ${\rm WL}(Y)=1$ and ${\rm WL}(\map_{*} (X,Y;f))>1$. The proof of Proposition \ref{prop:4.6} enables us to conclude that $\nil B\geq \omega$ and that $\omega \geq 3$ since $V=\Ker d_{1}$. Moreover we have
$$
\wl (\map_{*}(X,Y;f))\leq \frac{1}{\omega  -1}(\nil B -1)+1\leq \nil B-1.
$$
Thus our upper bound of the Whitehead length of the mapping space may be less than that described in Theorem \ref{thm:4.5}.
}\end{rem}


\section{Computational examples}
We shall determine the Whitehead length of the mapping space from $\C P^{\infty }\times \C P^{n}$ to $\C P^{\infty }_{\Q}\times \C P^{m}_{\Q}$.
For this, we first compute the homotopy group of the mapping space. Recall that the CDGAs $(\Q [z_{2}],0)$ and $(\Lambda (x_{2},x'_{2n+1}), \ dx'_{2n+1}=x_{2}^{n+1})$ are minimal Sullivan models for $\C P^{\infty}$ and $\C P^{n}$, respectively. Here, $|x_{2}|=|z_{2}|=2$ and $|x_{2n+1}'|=2n+1$. Since $\C P^{m}$ is formal, that is the CDGA map $\rho  $
$$
(\Lambda (x_{2},x'_{2m+1}), \ dx'_{2m+1}=x_{2}^{m+1})\longrightarrow (\Q[y_{2}]/(y_{2}^{m+1}),0)=H^{*}(\C P^{m};\Q)
$$
defined by $\rho (x_{2})=y_{2}$, $\rho (x'_{2m+1})=0$ is a quasi-isomorphism, so the CDGA $(\Q[w_{2}]\otimes \Q[y_{2}]/(y_{2}^{m+1}),0)$ is a model for $\C P^{\infty }\times \C P^{m}$.
\begin{prop}\label{ex:4.4}
Let $k\geq 2$ and $m<n$. Then 
$$
\pi_{k}(\map (\C P^{\infty }\times \C P^{n}, \C P^{\infty }_{\Q}\times \C P^{m}_{\Q}; f))
\hspace{-0.1em}=\hspace{-0.1em}
\left\{
\begin{array}{cl}
\Q &(k=2 \ \text{and} \ q_{2}\neq 0)\\
\Q\oplus \Q & (k=2 \ \text{and} \ q_{2}= 0)\\
\hspace{-0.6em} \displaystyle\bigoplus _{0\leq i=n-m-l+1}^{n-l+1}\hspace{-2em}\Q & (k=2l-1, \ 2\leq l\leq n+1) \\
0&(\text{otherwise})
\end{array}
\right.
$$
Here, $f$ is the realization of the CDGA  map $\overline{f}$
$$
M(\C P^{\infty }\times \C P^{n})=\Q[z_{2}]\otimes \Lambda (x_{2},x_{2n+1}')\to \Q[w_{2}]\otimes \Lambda (x_{2},x_{2m+1}')=M(\C P^{\infty }\times \C P^{m})
$$
defined by
$
\overline{f}(z_{2})=q_{1}(w_{2}\otimes 1), \ \overline{f}(x_{2})= q_{2}(w_{2}\otimes 1) + q_{3}(1\otimes x_{2}) \  \text{and} \ \overline{f}(x_{2n+1}')= 0.
$
for some $q_{1},q_{2},q_{3}\in \Q$.
\end{prop}
\begin{proof}
We put $\Der^{n}=\Der ^{n}(\Q[z_{2}]\otimes \Lambda (x_{2},x_{2n+1}'), \ \Q[w_{2}]\otimes \Q[y_{2}]/(y_{2}^{m+1});\rho \overline{f})$ for convenience. For any elements $\theta _{r,s}\in \Der^{-2}$, we may write
$$
\theta_{r,s} (z_{2})=r, \ \theta_{r,s}(x_{2})=s \ \text{and} \ \theta_{r,s}(x'_{2n+1})=0
$$ 
for some $r,s\in \Q$. Then, 
$$
\partial \theta_{r,s}(z_{2})=\partial \theta_{r,s}(x_{2})=0, \ \partial \theta_{r,s}(x'_{2n+1})=-ns(\sum_{i+j=n}q_{2}^{i}q_{3}^{j}w_{2}^{i}\otimes y_{2}^{j}) .
$$
When $q_{2}\neq 0$, we see that $\theta_{r,s} $ is a cycle if and only if $s=0$, that is all cycles of $\Der^{-2}$ generated by $\theta _{1,0}$. When $q_{2}=0$, $\theta_{r,s}(x'_{2n+1})=0$ since $y^{n}_{2}=0$. Hence, $\theta_{1,0}$ and $\theta _{0,1}$ are generators of all cycles of $\Der^{-2}$.
In general, $\Der^{-2l}=0$ for $l\geq 2$ by degree reasons. It follows that
$$
\pi_{2l}(\map (\C P^{\infty }\times \C P^{n}, \C P^{\infty }_{\Q}\times \C P^{m}_{\Q}; f))\cong H^{-2l}(\Der^{*})=0 \ (l\geq 2).
$$
For any $\theta\in \Der ^{-2l+1}$, we may write
$$
\theta (z_{2})=0, \ \theta(x_{2})=0 \ \text{and} \ \theta(x'_{2n+1})=\sum_{i=0}^{n-l+1}r_{i}w_{2}^{i}\otimes y_{2}^{n-l+1-i}.
$$
Note that if $l>n+1$, $\Der^{-2l+1}=0$ by degree reasons. It is easily seen that all elements of $\Der^{-2l+1}$ are cycles. Moreover, we see that $y^{n-l+1-i}=0$ if and only if $0\leq i\leq n-m-l$. Therefore, we have
$$
\pi_{2}(\map (\C P^{\infty }\times \C P^{n}, \C P^{\infty }_{\Q}\times \C P^{m}_{\Q}; f))\cong H^{-2}(\Der^{*}) \hspace{-0.1em} \cong \hspace{-0.1em}
\left\{
\begin{array}{cl}
\hspace{-0.5em}\Q              &\hspace{-0.5em} (k=2 \ \text{and} \ q_{2}\neq 0)\\
\hspace{-0.5em}\Q\oplus \Q &\hspace{-0.5em} (k=2 \ \text{and} \ q_{2}= 0)
\end{array}
\right.
$$
and
$$
\pi_{2l-1}(\map (\C P^{\infty } \times \C P^{n}, \C P^{\infty }_{\Q}\times \C P^{m}_{\Q}; f))
 \cong H^{-2l+1}(\Der^{*})=0 \ (l>n+1).
$$
\end{proof}

\begin{prop}
Let $m<n$. Then one has
$$
\wl (\map (\C P^{\infty }\times \C P^{n}, \C P^{\infty }_{\Q}\times \C P^{m}_{\Q}; f))=
\left\{
\begin{array}{ll}
\hspace{-0.3em}2&\hspace{-0.3em}(n-m=1, \ q_{2}=0, \ q_{3}\neq 0 )\\
\hspace{-0.3em}1&\hspace{-0.3em}(\text{otherwize})
\end{array}
\right.
$$
\end{prop}
\begin{proof}
By the definition of the bracket in $H^{*}(\Der^{*})$, we see that if $\varphi ,\psi \in H^{\leq -3}(\Der^{*})$, then $[\varphi ,\psi ]=0$ since $\varphi (x_{2})=0$ and $\psi (x_{2})=0$. That is $[\varphi ',\psi' ]\neq 0$ means $|\varphi '|=|\psi '|=-2$. It shows that 
$$
\wl(\map (\C P^{\infty }\times \C P^{n}, \C P^{\infty }_{\Q}\times \C P^{m}_{\Q}; f))\leq 2.
$$
If $q_{2}\neq 0$, by Proposition \ref{ex:4.4}, $H^{-2}(\Der^{*})$ is generated by $\theta_{1,0}$. The equality $[\theta_{1,0},\theta_{1,0}]=0$ shows that $\wl(\map (\C P^{\infty }\times \C P^{n}, \C P^{\infty }_{\Q}\times \C P^{m}_{\Q}; f))=1$. On the other hand, if $q_{2}=0$, $\theta _{0,1}$ is a generator of $H^{-2}(\Der^{*})$ and
$$[\theta _{0,1},\theta_{0,1}](x'_{2n+1})=q_{3}^{n-1}y_{2}^{n-1}.
$$
This completes the proof.
\end{proof}

\end{document}